\documentclass[12pt]{article}

\usepackage{a4wide, amsmath,amssymb,amscd,amsthm,makeidx,txfonts,graphicx,url,bm}
\usepackage{tikz,float}
\usetikzlibrary{positioning}
\usetikzlibrary{calc}

\usepackage[osf]{Baskervaldx} 

\usepackage{authblk}

\usepackage{color}
\definecolor{orange}{rgb}{.9,.6,.2}
\definecolor{violet}{rgb}{.5,0,.5}
\definecolor{dg}{rgb}{0,0.67,0}

\definecolor{cof}{RGB}{219,144,71}
\definecolor{pur}{RGB}{186,146,162}
\definecolor{greeo}{RGB}{91,173,69}
\definecolor{greet}{RGB}{52,111,72}

\renewcommand{\bar}{\overline}

\newcommand{\nc}{\newcommand}

\nc{\noi}{\noindent}
\nc{\cmmt}[1]{}

\newcounter{projectcnt}
\nc{\sctn}[1]{{\bigskip\addtocounter{projectcnt}{1}%
\begin{center}{\textbf{Project \theprojectcnt. #1}}\end{center}\medskip}}
\nc{\plainsctn}[1]{{\bigskip\begin{center}{\textbf{#1}}\end{center}\medskip}}

\nc{\dave}[1]{\begin{quote}\em #1\end{quote}}
\nc{\bbP}{{\bf{P}}}
\nc{\N}{{\bf{N}}}
\nc{\bF}{{\bar{\bf{F}}}}
\nc{\Gdave}{{\bf{G}}}
\nc{\I}{{\bf{I}}}
\nc{\cM}{{\cal M}}
\nc{\cK}{{\mathcal K}}
\nc{\cA}{{\cal A}}
\nc{\cO}{{\mathcal O}}
\nc{\cS}{{\mathcal S}}
\nc{\cNF}{{\cal NF}}
\nc{\cMF}{{\cal MF}}
\nc{\cLS}{{\cal LS}}
\nc{\Aff}{\mbox{Aff}}
\nc{\Mod}{{\mbox{\textup{Mod}}}}
\nc{\T}{\mathbb{T}}
\nc{\nf}{S_k^{\textup{\tiny new}}(N, \chi)}
\nc{\of}{S_k^{\textup{\tiny old}}(N, \chi)}
\nc{\mf}{S_k(N, \chi)}
\nc{\Skip}{\mathrm{\rm Skip}}
\nc{\Eis}{\mathrm{Eis}}
\nc{\Poly}{\mathrm{\rm Poly}}
\nc{\Polys}{\mathrm{\rm Polys}}
\nc{\Ext}{\mathrm{\rm Ext}}
\nc{\Ore}{\mathrm{\rm Ore}}
\nc{\Step}{\mathrm{\rm Step}}
\nc{\Mass}{\mathrm{\rm Mass}}
\nc{\mass}{\mathrm{\rm mass}}
\nc{\Vol}{\mathrm{\rm Vol}}

\def\sump{\operatornamewithlimits{\sum\raise7pt\hbox{$\prime$}}}

\DeclareMathOperator{\aff}{aff}

\def\sumpp{\operatornamewithlimits{\sum\raise9pt\hbox{\kern-2pt\scriptsize$(p)$\kern-11pt}\kern6pt}}

\newcommand{\calF}{{\cal F}}

\newcommand{\Spec}{\mbox{Spec}}

\newcommand{\Z}{{\mathbb Z}}
\newcommand{\R}{{\mathbb R}}
\newcommand{\F}{{\mathbb F}}

\newcommand{\C}{{\mathbb C}}

\newcommand{\A}{\mathbb A}

\newtheorem{theorem}{Theorem}[section]

\newcommand{\litem}{\par\noindent\dimen0=\parindent%
    \advance\dimen0 by-4pt
               \hangindent=\dimen0\ltextindent}

\newcommand{\ltextindent}[1]{\hbox to \hangindent{#1\hss}\ignorespaces}
\newcommand{\ltextjndent}[1]{\hbox to \hangindent{#1\hss}\ignorespaces\kern-1ex}

\newtheorem{example}[theorem]{Example}

\begin{document}

\title{Remarks on polynomial count varieties}

\author{Fernando Rodriguez Villegas}
\affil{ICTP \\
{\tt villegas@ictp.it}}
\author{Nicholas M. Katz}
\affil{Princeton University \\
{\tt nmk@math.princeton.edu}}

\maketitle
\section{}
In this short note we prove a couple of facts about polynomial count
varieties, answering natural questions that they raise. A {\it
  polynomial count} $X$ variety is essentially one for which its
number of points over finite fields $\F_q$ is given by a
polynomial~$\#X(\F_q)=C(q)$. Well-known examples include affine or
projective space (or more generally the Grassmanian) and other
standard varieties. For the general definition of ``polynomial count", and the connection of
$C(q)$ with the mixed Hodge structure of $X$ see~\cite{HRV}[Appendix, pp. 617-618].
There we began with $X/\C$ a separated scheme of finite type, and a polynomial $C[t] \in \Z[t]$.
We said that $X/\C$ was  ``polynomial count", with polynomial $C(t)$, if
 for some ``spreading out" $X_R/R$ of $X/\C$ to  a subring $R \subset \C$ which is finitely
generated as a $\Z$-algebra, we had:

For every finite field $\F_q$ and every ring homomorphism  $\phi: R \rightarrow \F_q$,
$$\#X_R(\F_q)=C(q).$$

Notice that even if our $X/\C$ has a descent to $\Z$ and is polynomial count, it may not
be polynomial count over $\Z$ in any naive sense. Simplest example: $x^2+1=0$ is polynomial
count over $\Z[i,1/2]$, with $C$ the constant polynomial $2$, but it is not polynomial
count over $\Z$.

In this paper, we will be concerned with the following situation. 
We are given a separated scheme $X/\Z$  of finite type over $\Z$, and $C[t] \in \Z[t]$ a polynomial.
Given an integer $D \ge 1$, 
We will say that $X/\Z$ is polynomial count ``outside $D$" if, for every finite field $\F_q$ in which $D$ is invertible, we have
$$\#X(\F_q)=C(q).$$
We will say $X/\Z$ is polynomial count if there exists such a $D$.

The two questions we address are the following
\begin{itemize}
\item
 1) If $X/C$ is smooth, polynomial count with $C(t)=t^n$ for some
  $n$, is $X$ isomorphic to affine space $\A^n$?
\item 2) If $X/C$ is polynomial count, is it true that its Hodge
  numbers in a given graded piece of fixed weight satisfy~$h^{p,q}=0$
  unless $p=q$?
\end{itemize}

We show that in both cases the answer is no.

\section{}
Let $H\in \C[x_1,\ldots,x_N]$ be a polynomial with
$N\geq 1$.  To state the main result in this section we need the
following combinatorial constants. For a subset $S\subseteq
\{1,\ldots,N\}$ let
$\Delta_S$ be the Newton polytope of the polynomial obtained by
setting $x_i=0$ in $H$ for every $i\notin
S$, if the specialization is non-zero. Otherwise, set
$\Delta_S=\emptyset$ if the specialization is zero. Then define
\begin{equation}
\label{f-defn}
f_{r,n}:=\#\{S \,|\, \#S=r, \dim \Delta_S=n-1\},
\end{equation}
with $\dim \Delta_S=-1$ if $\Delta_S=\emptyset$.

Let 
$$
\calF_\Delta(x,y):=\sum_{r,n=0}^Nf_{r,n}x^ry^n
$$
be the placeholder polynomial for these constants.

For convenience we let $\sigma_n\subseteq \Z^n$ be the standard
simplex of dimension $n-1$, convex hull of
$(1,0,\ldots,0),(0,1,0,\ldots,0),\ldots,(0,\ldots,0,1)$.  We will say
two lattice polytopes in $\Z^n$ are \emph{equivalent} if there is an
invertible affine transformation over $\Z$ taking one to the other.

\begin{theorem}
\label{simpl-count-thm}
  Let $X\subseteq \A^N$ be the zero locus over $\C$ of a polynomial
  $H\in \Z[x_1,\ldots,x_N]$ with $N\geq 1$ whose Netwon polytope $\Delta$
  is equivalent to $\sigma_N$. Let
$$
X':=X\cap T,
$$
where 
$$
T:=\Spec(\Z[x_1,x_1^{-1},\ldots,x_N,x_N^{-1}])
$$
is the torus. 

Then the varieties $X'$ and  $X$ are polynomial count outside $D$, for
$$D:={\rm\ the \ product\ of \ all\  nonzero\  coefficients\ of\ }H.$$
More
precisely, 

\begin{equation}
\label{simpl-count-0}
\#X'(\F_q)= c_N(q), 
\end{equation}
and
\begin{equation}
\label{simpl-count}
\#X(\F_q)=\sum_{r=0}^N\sum_{n=0}^N
f_{r,n}c_n(q)(q-1)^{r-n},
\end{equation}
where
$$
c_n(q)=
\begin{cases}
\sum_{i=1}^n(-1)^{n-i}\binom{n}{i}\, q^i + (-1)^n&\quad n\geq 0\\
1 & \quad n=0
\end{cases},
$$
 and the $f_{r,n}$ are defined in~\eqref{f-defn}.
\end{theorem}
\begin{proof}
  It is easy to verify that for $N\ge 1$
$$
\#\{(x_1,\ldots,x_N)\in (\F_q^\times)^N \,|\,
x_1+\cdots
+x_N=0\}= c_N(q)=\sum_{n=1}^N(-1)^{N-n}\binom Nn\, q^{n-1} + (-1)^N
$$
by M\"obius inversion as
$$
\sum_{n=0}^N\binom N nc_n(q)=q^{N-1},\qquad N\geq 1.
$$
Since $\Delta$ is equivalent to the standard
simplex~\eqref{simpl-count-0} follows. Indeed, 
let $y_i:=x^{m_i}$, where $m_1,\ldots,m_N$ are
the vertices of $\Delta$. Given the equivalence of $\sigma_N$ and
$\Delta$ this yields an automorphism $(x_1,\ldots,x_N)\mapsto
(y_1,\ldots,y_n)$ of the torus. But then the two equations
$$
x^{m_1}+\cdots+x^{m_N}=0, \qquad \qquad y_1+\cdots+y_N=0
$$
have the same number of solutions in $\F_q^\times$.

To prove~\eqref{simpl-count} we proceed recursively. Notice that
specializing $H(x_1,\ldots,x_N)$ by setting $x_i=0$ for $i\notin S$
with $S\subseteq \{1,\ldots,N\}$ results in a polynomial
$H_S(x_1,\ldots,x_N)$ whose Newton polynomial $\Delta_S$ is a face of
$\Delta$ (or empty). Therefore, $\Delta_S$ is also equivalent to
$\sigma_n$ for some $n\geq 0$ or empty (if $H_S$ vanishes
completely). In particular, if $\Delta_S$ is non-empty the number of
solutions of $H_S(x_1,\ldots,x_N)$ with $x_i\in \F_q^\times$ is
$c_n(q)(q-1)^{r-n}$, where $r:=\#S$. On the other hand, if $\Delta_S$
is empty then the number of solutions is $(q-1)^r$. The claim follows.
\end{proof}
\begin{example}
Let $n_1,\ldots,n_r$ be a tuple of pairwise coprime positive integers 
and let
$$
X:\quad x_1^{n_1}+\cdots + x_r^{n_r}=0.
$$
Then $X/\Z$ is polynomial count with $D=1$: more precisely,
\begin{equation}
\label{diag-count}
X(\F_q)=q^{r-1}.
\end{equation}
\begin{proof}
To see this note that the coprimality condition guarantees that the
Newton polytope $\Delta$ is equivalent to $\sigma_r$. Indeed,
translating by $-(n_1,0,\ldots,0)$ the polytope $\Delta$ is equivalent
to the simplex with vertices
$$
(0,\ldots,0),(-n_1,n_2,0,\ldots,0),\ldots (-n_1,0,\ldots,0,n_r).
$$
Let
$$
M:=\left(
\begin{array}{ccccc}
-n_1&-n_1&\cdots&&-n_1\\
n_2&0&\cdots&&0\\
0&n_3&0&\cdots&0\\
&&\vdots&&\\
0&\cdots&&0&n_r
\end{array}.
\right)
$$
To prove the claim is enough to show that we may find a column vector
$x:=(x_1,\ldots,x_r)$ with $x_i\in \Z$ such that the matrix
$U:=\left(x\,|\,M\right)\in \Z^{r\times r}$, adding $x$ to $M$ as its
first column, is invertible. Given such a $U$ we get
$$
M=UM',
$$
where
$$
M':=\left(
\begin{array}{ccccc}
0&0&\cdots&&0\\
1&0&\cdots&&0\\
0&1&0&\cdots&0\\
&&\vdots&&\\
0&\cdots&&0&1
\end{array}
\right)
$$
and since the simplex with vertices
$$
(0,\ldots,0),(0,1,0,\ldots,0),\ldots (0,\ldots,0,1)
$$
is equivalent to $\sigma_r$ the same is true of $\Delta$.

To find $x$ we need to solve for integers $x_i$ such that
$$
\det(U)=-d_rx_1+\cdots (-1)^rd_1x_r=\pm 1,\qquad d_i:=n_1\cdots \widehat{n_i}\cdots n_r.
$$
By the coprimality assumption the vector $(d_1,\ldots,d_n)$ is
primitive and hence such a solution exists.

It is also clear that $\calF_\Delta=\calF_{\sigma_r}$. The claim now follows
from Theorem~\ref{simpl-count-thm}.
\end{proof}
\end{example}
\begin{example}
Consider the Russell threefold~\cite{MR2327241}
$$
X: \quad x^2y+x+z^2+t^3=0.
$$
Its Newton polytope is the $3$-dimensional simplex in $\Z^4$ with
vertices the columns of the matrix
$$
\left(
\begin{array}{cccc}
2&1&0&0\\
1&0&0&0\\
0&0&2&0\\
0&0&0&3
\end{array}
\right)
$$
which is in fact equivalent to $\sigma_4$. We have
$$
\calF_\Delta(x,y)=x^4y^4 + 3x^3y^3 + x^3y^2 + 4x^2y^2 +
2x^2y + 3xy + x + 1.
$$
Then formula~\eqref{simpl-count} gives
$$
X(\F_q)=q^3.
$$
(That the Russell threefold is polynomial count over $\Z$, with $D=1$ and $C(t)=t^3$, can also be verified directly, see the next example)

\end{example}
It is known that $X(\C)$ is diffeomorphic to $\R^6$ but $X$ is not
isomorphic to $\A^3$. (See the discussion in~\cite{MR2327241}[\S1] and
the literature cited therein for more details on this type of
phenomena.) This answers our first question.

As another example, it is known that the fourfold
$$
x+x^2y+z^2+t^3+u^5 =0
$$
is diffeomorphic to $\R^8$ but not isomorphic to $\C^4$ as in
Russell's case (see ~\cite{MR2327241}[Remark 11.1]).

\begin{example}
  In general, if we let $n_1,\ldots,n_r$ be a tuple of pairwise coprime
  positive integers and consider the affine variety
$$
X:\quad x^2y+x+\sum_{i=1}^r x_i^{n_i}= 0
$$
in $\A^{r+2}$ with coordinates $(x,y,x_1,\ldots,x_r)$. 
Then $X$ is smooth over $\Z$ and 
$$
\#X(\F_q)=q^{r+1}.
$$
\begin{proof}
It is straightforward to check that $X$ is smooth. It follows easily
from the previous example that its Newton polytope is equivalent to
$\sigma_{r+2}$ (here again the coprimality assumption is crucial.)

It is in fact not hard to verify directly that $\#X(\F_q)=q^{r+1}$
without appealing to Theorem~\ref{simpl-count-thm}.  Indeed, for
$x\neq0$ we get $q^r(q-1)$ points by fixing $(x_1,\ldots,x_r)$ and
solving for $y$. For $x=0$ instead, we have $q^{r-1}$ points from
$\sum_{i=1}^rx_i^{n_i}=0$ by~\eqref{diag-count} times $q$ as $y$ is
arbitrary independent of $(x_1,\ldots,x_r)$.
\end{proof}

\end{example}

\section{}

Suppose X is polynomial count. Is it true that $h^{p,q;i}=0$ in weight
$i$ unless $p=q$?

This is not the case. Here is the simplest example.
Take an elliptic curve $E$, and denote by $E_{\aff}$ the affine curve which is
the complement of the origin. e.g., $E_{\aff}$ is the curve in $\A^2$ of literal
equation 
$$
E_{\aff}:\quad y^2 = f(x),
$$
with $f$ a cubic with distinct roots. 

Now consider the abstract scheme $X$ given by the disjoint union of
the affine variety $Y:=E_{\aff}\subseteq \A^2$, and the affine variety
$Z:=\A^2 \setminus E_{\aff}\subseteq \A^3$ defined by $z(y^2-f(x))=1$.

By the excision sequence for $Z$
$$
0=H^1_c(\A^2) \rightarrow H^1_c(Y) \rightarrow H^2_c(Z) \rightarrow
H^2_c(\A^2)=0.
$$
Hence
$$
H^1_c(Y) \cong H^2_c(Z).
$$
By the dual of Lefschetz affine, we have $H^1_c(Z)= 0$.

On the other hand, the excision sequence for $Y$ as $E \setminus 0$
shows that $H^1_c(Y) \cong H^1(E)$ and $H^2_c(Y) \cong H^2(E)$.  Thus
we conclude that
$$
gr_W^1(H^2_c(Z)) \cong H^1(E),
$$
so $Z$ has $h^{0,1;2}=h^{1,0,2} = 1$ and 
$h^{0,1;1}=h^{1,0;1} =0$ (because the entire $H^1_c$ vanishes).
Meanwhile, 
$$
gr_W^1(H^1_c(Y)) =H^1(E),
$$
so we conclude that $Y$ has $h^{0,1;1}=h^{1,0;1}=1$ and $Y$ has
$h^{0,1;2}=h^{1,0;2}=0$ (because $H^2_c \cong H^2(E)$ is of type
$(1,1)$).  So for $X$ we have $h^{0,1;1}=h^{1,0;1}=1$ and
$h^{0,1;2}=h^{1,0,2} = 1$. But this disjoint union is visibly
polynomial count, it has $q^2$ points.

Pick a point  $\alpha \in \A^2 \setminus E_{\aff}$, and a point
$\beta \in \A^3$. Then in $\A^5 = \A^2 \times \A^3$, we have the disjoint closed sets
$Y\times \beta$ and $\alpha \times Z$. In other words,
we can see our example as being a closed set $W\subseteq \A^5$. Then again playing 
with the excision sequence for $M:=\A^5 \setminus W$, we will get
$$
H^i_c(W) \cong  H^{i+1}_c(M),\qquad i \le 8,
$$
in particular for $i=1$ and $i=2$. Thus we find that $M$ is smooth and
irreducible and has $h^{0,1;2}=h^{1,0;2}=1$ and
$h^{0,1;3}=h^{1,0;3}=1$, though it too is trivially polynomial count,
it has $q^5 -q^2$ points.

\end{document}